\documentclass[12pt]{amsart}
\usepackage{amssymb,latexsym}
\usepackage[mathscr]{eucal}
\usepackage{graphics}
\usepackage{ulem}
\usepackage{verbatim}

\theoremstyle{plain}
\newtheorem{theorem}{Theorem}[section]
\newtheorem{corollary}[theorem]{Corollary}
\newtheorem{lemma}[theorem]{Lemma}
\newtheorem{proposition}[theorem]{Proposition}
\newtheorem{remark}[theorem]{Remark}
\newtheorem{definition}[theorem]{Definition}
\newtheorem{example}[theorem]{Example}

\def \co*{{\text{c}_{\text{o}}^*}}
\def \D*{$\Delta_{1/2}$-condition}
\newcommand{\be}{\begin{equation}\label}
\newcommand{\ee}{\end{equation}}
\newcommand{\bq}{\begin{equation*}}
\newcommand{\eq}{\end{equation*}}
\newcommand{\ba}{\begin{align*}}
\newcommand{\ea}{\end{align*}}
\newcommand{\bp}{\begin{proof}}
\newcommand{\ep}{\end{proof}}
\newcommand{\bL}{\begin{lemma}\label}
\newcommand{\eL}{\end{lemma}}
\newcommand{\bP}{\begin{proposition}\label}
\newcommand{\eP}{\end{proposition}}
\newcommand{\bC}{\begin{corollary}\label}
\newcommand{\eC}{\end{corollary}}
\newcommand{\bT}{\begin{theorem}\label}
\newcommand{\eT}{\end{theorem}}
\newcommand{\bR}{\begin{remark}\label}
\newcommand{\eR}{\end{remark}}
\newcommand{\bD}{\begin{definition}\label}
\newcommand{\eD}{\end{definition}}

\begin{document}

\title{Second Order Arithmetic Means \linebreak
  in Operator Ideals}
\author{Victor Kaftal}
\address{University of Cincinnati\\
          Department of Mathematics\\
          Cincinnati, OH, 45221-0025\\
          USA}
\email{victor.kaftal@math.uc.edu}
\author[Gary Weiss]{Gary Weiss*}\thanks{*Both authors were partially supported by grants from The Charles Phelps Taft Research Center.}
\email{gary.weiss@math.uc.edu}


\begin{abstract}
We settle in the negative the question arising from \cite{DFWW} on 
whether equality of the second order arithmetic means of two 
principal ideals implies equality of their first order arithmetic 
means (second order equality cancellation) and we provide 
fairly broad sufficient conditions on one of the principal ideals for this 
implication to always hold true. We present also sufficient 
conditions for second order inclusion cancellations. These conditions 
are formulated in terms of the growth properties of the ratio of 
regularity sequence associated to the sequence of s-number of a 
generator of the principal ideal. These results are then extended to 
general ideals.
\end{abstract}

\maketitle


\section{\leftline{\bf Introduction}}\label{S: 1}
Operators ideals, the two-sided ideals of the algebra $B(H)$ of bounded linear operators on a separable Hilbert space $H$, 
have played an important role in operator theory and operator algebras since they were first studied by J. Calkin \cite {jC41} in 1941. 
One of the recurring themes in this subject, from the early years on, was the study of commutator spaces, also called commutator ideals. 
The introduction of cyclic cohomology in the early 1980's by A. Connes (e.g., see \cite {aC94}) and the connection with algebraic K-theory 
by M. Wodzicki in the 1990's provided a powerful motivation for further work on operator ideals and commutator spaces. 
This work culminated in \cite{DFWW} with the full characterization of commutator spaces  in terms of arithmetic mean operations on ideals.

Arithmetic means were first connected, albeit only implicitly, to operator ideals in the study of the commutator space of the trace class in \cite{gW75} 
(see also \cite {gW80}, \cite  {gW86}) and then explicitly in \cite{nK89}. 
In \cite{DFWW}, arithmetic means provided a full characterization of the commutator space of arbitrary ideals
 and led to the introduction of a number of arithmetic mean ideals derived from a given ideal $I$, among which $I_a$  (see next section for the formal definitions). 
For instance, an ideal $I$ coincides with its commutator space (equivalently, it supports no nonzero trace) if and only if it is arithmetically mean stable, that is, $I=I_a$  \cite[Theorem 5.6]{DFWW}.
The study of the properties of $I_a$ and of other related arithmetic mean ideals and the application of these properties to operator 
ideals are the focus  of a program  by the authors of this paper which was announced in \cite{vKgW02} and includes 
\cite{vKgW04-Traces}-\cite{vKgW04-Majorization} and this paper.

The operation $I \to I_a$ is inclusion preserving and the  arithmetic 
mean cancellation properties for inclusion and equality are deeply 
linked to the structure of operator ideals. For instance in 
\cite{vKgW04-Density} we characterized when an ideal $I$ has the 
following first order arithmetic mean inclusion and equality 
cancellation properties:

\begin{itemize}
\item[(i$'$)] $J_a \subset I_a \Rightarrow J \subset I$ if and only if 
$I$ is am-closed (\cite[Lemma 6.1(C)]{vKgW04-Density})
\item[(ii$'$)] $J_a \supset I_a \Rightarrow J \supset I$ if and only if 
$I = \widehat I$ (\cite[Definition 6.10]{vKgW04-Density} and 
preceding remark)
\item[(iii$'$)] If $I$ is principal, then $J_a = I_a \Rightarrow J = I$ 
if and only if $I$ is am-stable (\cite[Theorem 6.7(A)]{vKgW04-Density}) \\
\end{itemize}

For the definitions, notations and preliminary properties, see 
Section \ref{S: 2}.  Notice first that although the terminology we 
use here is new and due to \cite{DFWW}, the inclusion $J_a \subset 
I_a $ in the case where both $I$ and $J$ are principal ideals has a 
time honored history: it reduces to the (infinite) majorization of 
the s-number sequences of the generators (up to a normalization 
scalar) $\eta \prec \xi$ in the sense that $\sum_1^n\eta_j \le 
\sum_1^n\xi_j$ for all $n$, namely $\eta_a \le \xi_a$ (see  for 
instance \cite{HLP52}, \cite {GK69}, and \cite {MO79}).

If $I$ is principal, then conditions in (i$'$) ($I$ is am-closed) and 
in (iii$'$) ($I$ is am-stable) are both equivalent to the regularity of 
the sequence $\eta$ of the s-numbers of a generator of $I$, while the 
condition in (ii$'$) is strictly stronger and is derived from  the 
construction of a new principal arithmetic mean ideal  $\widehat I$ 
(see Section 2). Notice  that the first order equality cancellation 
property (iii$'$) is stated directly for principal ideals. Indeed, 
am-stability is not a sufficient condition for (iii$'$) even in the 
case when $I$ is countably generated (\cite [Example 
5.5]{vKgW04-Density}) nor do we  know of any natural necessary and 
sufficient condition for the general ideal case.

Second order arithmetic mean cancellations are considerably more complex 
even for principal ideals and are the focus of this paper. The 
questions we address here are: which conditions on an ideal $I$ 
guarantee that the following second order arithmetic mean inclusion 
cancellations and equality cancellation hold?
\begin{itemize}
\item[(i)] $J_{a^2} \subset I_{a^2} \Rightarrow J_a \subset I_a$
\item[(ii)] $J_{a^2} \supset I_{a^2} \Rightarrow J_a \supset I_a$
\item[(iii)] $J_{a^2} = I_{a^2} \Rightarrow J_a = I_a$\end{itemize}

The first natural ``test" question, which was posed by  M. Wodzicki, 
arising from work in \cite{DFWW}, is whether equality cancellation 
(iii) holds automatically for all pairs of principal ideals. 
Reformulated in terms of the s-number sequences $\xi$ and $\eta$ of 
the generators of the two principal ideals, the question asks whether 
the equivalence $\xi_{a^2}\asymp \eta_{a^2}$ of the sequences of the second order arithmetic means 
  always implies the equivalence $\xi_a \asymp \eta_a$ of the first order means.

The answer to this question is negative and is presented in Example 
\ref {E: Example 6}, one of the main results in this paper. The 
intuition behind the construction of this example led to the notion 
(Definition \ref{D: r}) of the \textit{ratio of regularity} sequence 
$r(\xi): = \frac{\xi_a}{\xi}$ for a nonincreasing sequence $\xi \in \text{c}_{\text{o}}$  
and then, indirectly, to the other results in this paper.

The second natural question is whether, at least when $I$ is 
principal, the arithmetic stability of $I$ (namely, the regularity of 
the sequence $\eta$ of s-numbers of a generator of $I$) might be the 
necessary and sufficient condition for (iii), as is the case for 
first order equality cancellation. We found that while regularity 
is indeed sufficient, it is ``very far" from being necessary, where 
``very far" is meant in terms of the ratio of regularity 
$r(\eta_a)=\frac{\eta_{a^2}}{\eta_a} $ of the sequence $\eta_a$.

Indeed, as is easy to show,  $1\le r(\eta_a)_n \le \log n$ for all 
$n >1$ and all $\eta$ (Lemma \ref{L: H_n-a^2/a uparrow}). The two 
``extremal cases" for $r(\eta_a)$  are thus when $r(\eta_a)$ is 
bounded, which is precisely the case when $\eta$ is regular (see 
Corollary \ref {C: g regular iff a(g)} and preceding discussion) and 
when $r(\eta_a)\asymp \log$. The latter condition is equivalent to 
what we call the exponential $\Delta_2$-condition 
$\underset{m}{\sup}\,\frac{m^2(\eta_a)_{m^2}}{m(\eta_a)_m} < \infty$
   (Proposition \ref{P: exp Delta2 equiv log asymp a2/a}).
Surprisingly (for us) it turned out that either of these two extremal 
cases,  $\eta$ regular or $r(\eta_a)\asymp \log$,  is sufficient for 
$I$ to have the second order equality cancellation property (iii) 
(Theorem \ref{T: equality cancel}(i)).

A further investigation of the exponential $\Delta_2$-condition 
shows that if $\eta$ satisfies this condition, then so does 
$\eta_{a^p}$ for all $p\in \mathbb N$ (Corollary \ref {C:higher 
order}) and hence higher order cancellations also hold (Theorem \ref 
{T: higher order cancellation}).

While we do not know if these two conditions,  $\eta$  regular or 
$r(\eta_a)\asymp \log$, are necessary for the equality cancellation 
in (iii) to hold, we know that they are too strong for the inclusion 
cancellation in (i). Indeed a weaker  sufficient condition for (i) is 
that $r(\eta_a)$ is equivalent to a monotone sequence (Theorem \ref{T: monotone}(i)).

On the other hand, the two conditions, $\eta$  regular or 
$r(\eta_a)\asymp \log$, are too weak for the  inclusion cancellation 
in (ii). Indeed Example \ref {e: omega1/2} shows that the principal 
ideal generated by the regular sequence $\omega^{1/2}$ (where 
$\omega$ is the harmonic sequence $<\frac{1}{n}>$) does not satisfy 
the inclusion cancellation in (ii).

A sufficient condition for the inclusion cancellation in (ii) is 
that either $\eta = \widehat \eta$ or $r(\eta_a)\asymp \log$ (Theorem 
\ref{T: hat or log}(i), and for the definition of $\widehat \eta$ see the end of Section \ref{S: 2}).

Sufficient  conditions for each of the two cancellation properties 
(i) and (ii) in the case of  a general ideal $I$ are that every 
sequence in the characteristic set $\Sigma(I)$ of $I$ is pointwise majorized by 
a sequence satisfying the corresponding condition for the principal ideal case (Theorems \ref{T: monotone}(ii) and \ref{T: hat or log}(ii)). 
For the  equality cancellation property (iii) for general ideals however, 
we have to ask a stricter sufficient condition than for the principal 
ideal case, as the proof for the ``weaker" condition fails (Theorem \ref{T: equality cancel}(ii)). 
This corresponds roughly to the fact that  first order equality 
cancellation (iii$'$) can fail for arithmetically mean stable ideals 
even for a countably generated ideal (\cite[Example 5.5]{vKgW04-Density}).

There is a dual theory to arithmetic mean ideals in the trace class 
implicit in \cite{DFWW} and developed explicitly in 
\cite{vKgW02}-\cite{vKgW04-Density}, called arithmetic mean at 
infinity ideals and we found strong parallels throughout their 
development. It seems likely that this parallel would continue with 
the results of this paper although we have not pursued this line of 
investigation.

\section{\leftline{\bf Preliminaries and notation}} \label{S: 2}

Calkin \cite{jC41} established a correspondence between the two-sided 
ideals of $B(H)$ for a complex-valued infinite-dimensional separable 
Hilbert space $H$ and the \textit{characteristic sets}. 
These are the positive cones of $\text{c}_{\text{o}}^*$
(the collection of sequences decreasing to $0$) that are hereditary 
(i.e., solid) and invariant under \textit{ampliations}
\[
\text{c}_{\text{o}}^* \owns \xi \rightarrow 
D_m\xi:=~<\xi_1,\dots,\xi_1,\xi_2,\dots,\xi_2,\xi_3,\dots,\xi_3,\dots>
\]
where each entry $\xi_i$ of $\xi$ is repeated $m$-times.
The order-preserving lattice isomorphism
$I \rightarrow \Sigma(I)$ maps each ideal to its characteristic set 
$\Sigma(I) := \{s(X) \mid X \in I\}$ where $s(X)$ denotes the 
sequence of
$s$-numbers of $X$, i.e., all the eigenvalues of $|X| = (X^*X)^{1/2}$ 
repeated according to multiplicity, arranged in decreasing order, and 
completed by adding infinitely many zeroes if $X$ has finite rank.
Moreover, for every characteristic set $\Sigma \subset 
\text{c}_{\text{o}}^*$,
if $I$ is the ideal generated by $\{diag~\xi \mid \xi \in \Sigma\}$
where $diag~\xi$ is the diagonal matrix with entries 
$\xi_1,\xi_2,\dots$, then we have $\Sigma = \Sigma(I)$.

If $\xi \in \co* $, denote by $\xi_a$ the arithmetic mean sequence of 
$\xi$, namely
\[
(\xi_a)_n = \frac{1}{n}\sum_1^n \xi_j.
\]

If $I$ is an ideal, then the arithmetic mean ideals $_aI$ and $I_a$, 
called the \textit{pre-arithmetic mean} and \textit{arithmetic mean} 
of $I$, are the ideals with characteristic sets
\[
\Sigma(_aI) := \{\xi \in \text{c}_{\text{o}}^* \mid \xi_a \in \Sigma(I)\},
\]
\[
\Sigma(I_a) := \{\xi \in \text{c}_{\text{o}}^* \mid \xi = 
O(\eta_a)~\text{for some}~ \eta \in \Sigma(I)\}.
\]
The ideals $I^o:= (_aI)_a$ and $I^-:=\,_a(I_a)$ are called the 
\textit{am-interior} and \textit{am-closure} of $I$ and the following 
5-chain of inclusions  holds (all of which can be simultaneously proper):
\[
_aI \subset I^o \subset I \subset I^- \subset I_a.
\]
Simple consequences of the 5-chain are the identities 
$I_a=(_a(I_a))_a$ and $_aI=\,_a((_aI)_a)$
and the consequent idempotence of the maps $I \rightarrow I^-$ and $I \rightarrow I^o$.
And either derived from these or proven directly  are the higher order 5-chains of inclusions:
\[_{a^n}I \subset (_{a^n}I)_{a^n} \subset I \subset \,_{a^n}(I_{a^n}) 
\subset I_{a^n},\]
the higher order identities
$I_{a^n}=(_{a^n}(I_{a^n}))_{a^n}$ and $_{a^n}I=\,_{a^n}((_{a^n}I)_{a^n})$,
  and the idempotence of the maps
$I \rightarrow \,_{a^n}(I_{a^n})$ and $I \rightarrow (_{a^n}I)_{a^n}$.
So $_{a^n}(I_{a^n})$ are called higher order am-closures and form an 
increasing nest.
(Similarly, $(_{a^n}I)_{a^n}$ are called higher order am-interiors 
and form a decreasing nest.)

Principal ideals are those ideals generated by a single operator $X$ 
and we denote them by $(X) = (s_n(X))$, i.e., for every $\xi \in 
\co*$, the principal ideal $(diag\,\xi)$ is also denoted by $(\xi)$.  Then 
$\Sigma((\xi)) = \{\eta\in \co* \mid \eta = O(D_m\xi) \text{ for some 
} m\in \mathbb N\}$ and $(\xi)_a = (\xi_a)$.
Since  arithmetic mean sequences satisfy the $\Delta_{1/2}$ 
condition, i.e.,  $\xi_a \asymp D_2(\xi_a )$ and hence $\xi_a \asymp 
D_m(\xi_a )$ for all $m\in \mathbb N$, one has $\Sigma((\xi)_a) = \{\eta\in 
\co* \mid \eta = O(\xi_a) \}$. 
As a consequence,  $(\eta)_a \subset (\xi)_a$ 
(resp., $(\eta)_a = (\xi)_a$) if and only if $\eta_a = O(\xi_a)$ (resp., $\eta_a \asymp O(\xi_a)$).

We denote by $\omega$ the harmonic sequence $<\frac{1}{n}>$ and by 
$H$ the sequence $<\sum_1^n\frac{1}{j}>$ and we often use the inequalities
\be{e: Hn}
  \frac{1}{n} + \log n < H_n < 1 +  \log n \quad \text{for $n > 1$}
\ee
and
\be{e: Hn-Hm}
\log\,\frac{n}{m} - \frac{1}{m+1}  <  \log\,\frac{n+1}{m+1}  <  H_n - H_m  <  \log\,\frac{n}{m}  \quad \text{for $n > m$}.
\ee

Given an ideal $I$ and a sequence $\xi \in \co*$, the ideal $\widehat 
I$ and the sequence $\widehat \xi$ mentioned in the introduction are 
defined in \cite [Definitions 6.10, 6.11] {vKgW04-Density} as 
$\widehat I:= \cap\{J\mid J_a \supset I_a\}$ and $ (\widehat 
\xi)_n:=(\xi_{a})_{\nu(\xi)_{n}} $
where $\nu(\xi)_{n} := \min \{k \in \mathbb N \mid 
\sum_{i=1}^{k}\xi_{i} \geq n\xi_1 \}$. Then $\widehat 
{(\xi)}=(\widehat \xi) $ for $\xi$ not summable. In particular, we 
proved there that $\widehat {\omega^p}=\omega^{p'}$ where 
$\frac{1}{p}-\frac{1}{p'}=1$  and that if $\xi \asymp \widehat \xi$ 
then $\xi$ is regular (see [ibid., Corollaries 6.15 and 6.16, and 
Proposition 6.17]).

\section{\leftline{\bf Ratio of regularity }} \label{S: 3}

Ideals that coincide with their commutator space (i.e., those which do not support any nonzero trace) were identified in 
\cite{DFWW} as the arithmetic mean stable ideals, i.e., the ideals $I$ that coincide with their arithmetic mean $I_a$. 
Nonzero principal am-stable ideals are precisely the  ideals $I=(\xi)$ with a \textit{regular} generator, 
i.e., $\xi_a \asymp \xi$ (see \cite[III.14.3]{GK69}).  
Since $\xi \le \xi_a$ holds always, a sequence $\xi$ is regular precisely when $\xi \notin \Sigma(F)$ 
(i.e., $\xi_n >0$ for all $n \in \mathbb N$) and $\frac{\xi_a}{\xi}$ is bounded.  
In this paper we will see how properties of the ratio of regularity sequence $\frac{\eta_{a^2}}{\eta_a}$ of $\eta_a$ 
relate to second order arithmetic mean cancellations for $I=(\eta)$.

\begin{definition}\label{D: r}
Let $\xi \in \text{c}_\text{o}^*\setminus \Sigma(F)$. 
The sequence $r(\xi) := \frac{\xi_a}{\xi}$ is called  the ratio of regularity of $\xi$.
\end{definition}

\noindent   Notice that  for all $\xi \in 
\text{c}_\text{o}^*\setminus \Sigma(F)$, $r(\xi)_1=1,~r(\xi)_n \geq 
1$ for all $n$ and if $m$ is the first index for which $\xi_m < 
\xi_1$, then  $r(\xi)_n > 1$ for all $n \geq m$.

Ratios of regularity sequences have appeared implicitly in the literature and are helpful in analyzing various sequence properties. 
For instance they are instrumental in deriving the Potter-type inequality characterizing the sequences satisfying the $\Delta_{1/2}$- condition and characterizing regularity, see \cite[Theorem 3.10]{DFWW}, \cite[Proposition 4.14 (proof)-Corollary 4.15]{vKgW04-Traces}.  
Their usefulness derives in part from the inversion formula (\ref{e: reconstruction}) in the next proposition.

\bP{P:admiss}
There is a one-to-one correspondence between the sequences $\xi \in\co*\setminus \Sigma(F)$ with $\xi_1=1$ and the sequences $<r_n>$ 
with $r_1=1$,  $(n+1)r_{n+1} \geq  nr_n+1 $ for all $n$ and  $\log r_n +\sum_2^{n}\frac{1}{j}(1-\frac{1}{r_j})\rightarrow \infty$. 
This correspondence is given by the map $~\xi \to r(\xi):=\frac{\xi_a}{\xi}$ with inverse map $~r \to \xi~$ given by
\begin{equation}\label{e: reconstruction}
\xi_n=
\begin{cases}
1 &\text{for $n=1$}\\
\frac{1}{nr_n} \prod_{2}^{n}(1+\frac{1}{jr_j-1})= \frac{1}{nr_n \prod_{2}^{n}(1-\frac{1}{jr_j})} &\text{for $n>1$}.
\end{cases}
\end{equation}
\eP

\bp
The monotonicity of $\xi$ is equivalent to the inequality for $r_n$.
Indeed, first assume that $\xi$ is a monotone nonincreasing sequence with $\xi_1=1$  and $\xi_n >0$ for all $n$. 
Let $r:= r(\xi) = \frac{\xi_a}{\xi}$. 
Then  $r_1= \frac{(\xi_a)_1}{\xi_1} = 1$ and
\[
(n+1)r_{n+1}\xi_{n+1}  = \sum_{1}^{n+1} \xi_i = \sum_{1}^{n} \xi_i + \xi_{n+1} = nr_n\xi_n + \xi_{n+1},
\]
for all $n$, i.e., the sequences $r$ and $\xi$ satisfy the recurrence relation
\be {e:recurr}
((n+1) r_{n+1}-1)\xi_{n+1} = nr_n\xi_n, \quad \text{with  $~\xi_1=1$ 
and $r_1=1.$}
\ee
By the monotonicity of $\xi$,
\[
((n+1) r_{n+1}-1)\xi_{n+1} \ge nr_n\xi_{n+1}
\]
  and hence   $(n+1)r_{n+1} \geq  nr_n+1 $, since $\xi_{n+1} >0$ for all $n$.

Conversely, assume that $r=\,<r_n>$ is a sequence with $r_1 = 1$ and $(n+1)r_{n+1} \geq  nr_n+1 $ for all $n$ and $\xi$ is the sequence 
given by (\ref{e: reconstruction}). 
Then
\[
\frac{\xi_{n+1}}{\xi_n} = \frac{nr_n}{(n+1) r_{n+1}}\big(1+\frac{1}{(n+1) r_{n+1}-1} \big) = \frac{nr_n}{(n+1) r_{n+1}-1} \le 1,
\]
i.e.,  $\xi$ is monotone nonincreasing. 
Moreover, it is easy to verify that $\xi$ and $r$ also satisfy the recurrence relation (\ref{e:recurr}) and hence that $r(\xi)=r$.

It remains to prove that the limit condition on $r$ is necessary and sufficient for $\xi_n\to 0$. 
Clearly, $\xi_n\to 0$ if and only if
\be{e: conv0}
\log nr_n-\sum_{2}^{n}\log (1+\frac{1}{jr_j-1})   \rightarrow \infty.
\ee
  
\noindent Equivalently,
\[
\log r_n + \sum_{2}^{n}\big(\frac{1}{j}  - \log 
(1+\frac{1}{jr_j-1})\big)   \rightarrow \infty.
\]
Since $\frac{1}{x} < \log \big( 1 + \frac{1}{x-1}\big)  < 
\frac{1}{x-1}$ for $ x > 1$ and hence
\[
\sum_{2}^{n}\Big( \frac{1}{j}-\frac{1}{jr_j} \Big) - \sum_2^n \frac{1}{j(j-1)} 
\le \sum_{2}^{n}\Big( \frac{1}{j}-\frac{1}{jr_j-1} \Big)
< \sum_{2}^{n}\Big( \frac{1}{j}-\log(1+\frac{1}{jr_j-1} \Big)
< \sum_{2}^{n}\Big( \frac{1}{j}-\frac{1}{jr_j} \Big),
\]
it follows that (\ref{e: conv0})  is equivalent to
\be{e:admiss}
\log r_n + \sum_{2}^{n}\frac{1}{j}(1-\frac{1}{r_j}) \rightarrow \infty.
\ee
\ep

\bR{R:admiss}  Since the sequence in (\ref {e: conv0}) is monotone 
nondecreasing, sufficient conditions for  (\ref{e:admiss}) are $\sup r= \infty$ or $\sum_2^{\infty}\frac{1}{jr_j}<\infty$. 
A necessary and sufficient condition for (\ref {e:admiss}) is that $\sup r= \infty$ 
or  $ \sum_{2}^{\infty} \frac{1}{j}(1-\frac{1}{r_j}) = \infty$. (These conditions are not mutually exclusive.)
\eR
  \noindent Lower bound estimates for the rate of decrease of 
$\frac{\xi_n}{\xi_m}$ for $n \ge m$ are an important tool in the subject. 
For instance,  the Potter-type characterization of a regular 
sequence states that $\xi$ is regular if and only if $\xi_n \ge 
\big(\frac{m}{n}\big)^p\xi_m$ for some $0 < p < 1$ and for all $n\ge m$ (see for instance \cite[Proposition 4.14]{DFWW}). 
Since  $n(\eta_a)_n$ is monotone increasing for every $\eta \in\co*$ 
(strictly increasing if and only if $\eta_n > 0$ for all $n$, i.e., if $\eta \notin \Sigma(F)$), 
it follows that  $(\eta_a)_n \ge \big(\frac{m}{n}\big)(\eta_a)_m$, for all $n\ge m$. 
The next lemma provides both an identity and an upper bound estimate for the rate of decrease of 
$\frac{(\eta_a)_n}{(\eta_a)_m}$ for $n \ge m$, both of which are needed here.

\bL{L:ratio} Let $\xi \in\co*\setminus \Sigma(F)$ and  let  $r= r(\xi)$.
\item[(i)]  For all $n > m$,
\[
(\xi_a)_n = \frac{m}{n} 
\prod_{j=m+1}^{n}(1+\frac{1}{jr_j-1})(\xi_a)_m = \frac{ 
\frac{m}{n}}{\prod_{j=m+1}^{n}\big(1-\frac{1}{jr_j}\big)}(\xi_a)_m.
\]
\item[(ii)] If $\phi$ is a nondecreasing strictly positive sequence with $\phi\le r$, then
\[
(\xi_a)_n \le \big(\frac{m}{n}\big)^{1- \frac{2}{\phi_m}}(\xi_a)_m
\]
for every $m\in \mathbb N$ sufficiently large so that  $m \phi_m > 2$ 
and for every $n \ge m$.
\eL

\bp
\item[(i)] By applying (\ref{e: reconstruction}) to $\frac{1}{\xi_1}\xi$ we have
\[
\xi_n = \frac{1}{nr_n} \prod_{j=2}^{n}(1+\frac{1}{jr_j-1})\xi_1=
\frac{mr_m}{nr_n} \prod_{j=m+1}^{n}(1+\frac{1}{jr_j-1})\xi_m
=\frac{m}{n}\frac{\xi_n}{(\xi_a)_n}\prod_{j=m+1}^{n}(1+\frac{1}{jr_j-1})(\xi_a)_m.
\]
\item[(ii)]
Equality holds trivially for $n=m$, and for  $n > m$ one has
\begin{align*}
(\xi_a)_n
&= \frac 
{\frac{m}{n}}{\prod_{j=m+1}^{n}\big(1-\frac{1}{jr_j}\big)}(\xi_a)_m&\text{by 
(i)}\\
&\le\frac{ \frac{m}{n}}{\prod_{j=m+1}^{n}\big(1-\frac{1}{ 
j\phi_j}\big)} (\xi_a)_m&\text{since $\phi \le r$}\\
&\le \frac 
{\frac{m}{n}}{\prod_{j=m+1}^{n}\big(1-\frac{1}{j\phi_m}\big)} 
(\xi_a)_m&\text{because $\phi$ is nondecreasing}\\
&=  \frac{m}{n}e^{-\sum_{j=m+1}^{n}\log \big(1-\frac{1}{ 
j\phi_m}\big)}(\xi_a)_m\\
&\le  \frac{m}{n}e^{\sum_{j=m+1}^{n} \frac{2}{ j\phi_m}} (\xi_a)_m& 
\text{by the inequality $log\,(1-x) \ge -2x$ for $0 \le x < 
\frac{1}{2}$,} \\
&= \frac{m}{n}e^{\frac{2}{\phi_m}(H_n-H_m)}(\xi_a)_m\\
&\le \frac{m}{n}e^{\frac{2}{ \phi_m}\log \frac{n}{m} 
}(\xi_a)_m&\text{by (\ref{e: Hn-Hm}) }\\
&= \big(\frac{m}{n}\big)^{1- \frac{2}{\phi_m}}(\xi_a)_m.
\end{align*}
\ep

Notice that in general, while the ratio of regularity $\frac{\xi_a}{\xi}$ has downward variations bounded by the 
inequality $r_{n+1} \geq  \frac{n}{n+1}r_n+ \frac{1}{n+1} $, it can vary abruptly upwards since
\[
r(\xi)_{n+1} = \frac{(\xi_a)_{n+1}}{\xi_{n+1}}
= \frac{(n+1)(\xi_a)_n}{(n+1)\xi_{n+1}}
\ge  \frac{n(\xi_a)_n}{(n+1)\xi_{n+1}}
= \frac{n}{n+1}\frac{\xi_n}{\xi_{n+1}}r(\xi)_n
\]
and $\frac{\xi_n}{\xi_{n+1}}$ can be arbitrarily large.

Sequences in $\co*$ that are an arithmetic mean of another sequence 
in $\co*$ are however smoother and their ratios of regularity are 
subject to ``slower" upward variations. Indeed, if  $\xi = \eta_a$ for 
some $\eta \in \co*$,  the upward variation of $\frac{\xi_a}{\xi} = \frac{\eta_{a^2}}{\eta_a}$ is limited by the inequality
\be{e:ineq rn}
r(\eta_a)_{n+1}  = \frac{(n+1)(\eta_{a^2})_{n+1}}{(n+1)(\eta_a)_{n+1}}
\le \frac{(n+1)(\eta_{a^2})_n}{n(\eta_a)_n}
=(1+\frac{1}{n})r(\eta_a)_n.
\ee

As is easy to verify (cfr. \cite[Lemma 2.12]{vKgW04-Soft}), a 
sequence $\xi \in \co*$ is the arithmetic mean of another sequence in 
$\co*$ if and only if $\frac{\xi}{\omega}$ is nondecreasing and 
concave, i.e.,
\[2n\xi_n \ge (n+1)\xi_{n+1} +  (n-1)\xi_{n-1} ~\text{\quad for all}~ n>1.\] 
Define 
the \textit{concavity ratio} $c(\xi)$ of a sequence 
$\xi\in\co*\setminus\Sigma(F)$ to be
\be{e:concav}
c(\xi)_n:=\frac{n\xi_n}{(n+1)\xi_{n+1}}.
\ee
Concavity of $\frac{\xi}{\omega}$ is equivalent to the condition: $c(\xi)_n + \frac{1}{c(\xi)_{n+1}} \le 2$.
Since we are not going to make use of the concavity ratio beyond a slight improvement in inequality (\ref {e:ineq rn}) 
for Corollary \ref{C: concav2}(ii), we will only sketch briefly the proofs of the next proposition 
(which is an analog of Proposition \ref{P:admiss}) and its Corollaries \ref{C: concav1}, \ref{C: concav2}.

\bP{P:concav}
There is a one-to-one correspondence between the sequences $\xi 
\in\co*\setminus \Sigma(F)$ with $\xi_1=1$ and the sequences $<c_n>$ 
with  $c_n \ge 1 - \frac{1}{n+1}$ and $\sum_1^{\infty}\big(\frac{1}{j} + \log{c_j} \big) = \infty$.
  This correspondence is established by the map $\xi \to 
c(\xi):=~<\frac{n\xi_n}{(n+1)\xi_{n+1}}>$ with inverse map 
  \begin{equation}\label{e:reconstr c}
\xi_n=
\begin{cases}
1 &\text{for $n=1$}\\
\frac{1}{ n\prod_{1}^{n-1}c_j} &\text{for $n>1$.}
\end{cases}
\end{equation}
Moreover,
\[
c(\xi)_n=\frac{r(\xi)_{n+1}-\frac{1}{n+1}}{r(\xi)_n}\quad \text{and}\quad
r(\xi)_n=
\begin{cases}
1 &\text{ for $n=1$}\\
\frac{1}{n} + \sum_{k=1}^{n-1}\frac{1}{k}\prod_{j=k}^{n-1} c(\xi)_j 
&\text{ for $n>1.$}
\end{cases}
\]
\eP

\bp
It is straightforward to verify that (\ref{e:reconstr c}) provides the inverse of (\ref{e:concav}) and that the monotonicity of $\xi$ is 
equivalent to the condition  $c_n \ge 1 - \frac{1}{n+1}$. 
Also, $\xi_n \to 0$ if and only if 
$$\log{\frac{1}{\xi} }= \log n + \sum_1^{n-1}\log c_j \to \infty$$ 
which by (\ref {e: Hn}) is 
equivalent to $\sum_1^{\infty}\big(\frac{1}{j} + \log{c_j} \big) = \infty$. 
The remaining equations are simple computations.
\ep

\bC{C: concav1}
There is a one-to-one correspondence between the sequences $\xi = \eta_a$ with $\eta \in\co*$, $\eta_1=1$ and the sequences $<c_n>$ 
with $c_1 \ge \frac{1}{2}$, $c_n > 0$, $c_n+\frac{1}{c_{n+1}} \le 2$ 
for all $n\in \mathbb N$ and $\sum_1^{\infty}\big(\frac{1}{j} + \log c_j \big) = \infty$.
This correspondence is established by the map $\xi \to c(\xi)$ (Equation (\ref{e:concav})).
Moreover, $c(\xi)_n \uparrow 1$ and
  \begin{equation*}
\eta_n=
\begin{cases}
1 &\text{for $n=1$}\\
\frac{1-c(\xi)_{n-1}}{ \prod_{1}^{n-1}c(\xi)_j} &\text{for $n>1$.}
\end{cases}
\quad \text {and} \quad r(\eta)_n =
\begin{cases}
1 &\text{for $n=1$}\\
\frac{1}{n(1-c(\xi)_{n-1})} &\text{for $n>1.$}
\end{cases}
\end{equation*}
\eC

\bp
It is an immediate consequence of the above mentioned concavity of $\frac{\xi}{\omega}$ and of Proposition \ref{P:concav} 
that the sequence $c(\xi)$ satisfies the stated inequalities and the series condition when $\xi = \eta_a$ and $\eta \in \co*$.
Conversely, it is straightforward to show that if $<c_n>$ satisfies these conditions then it is nondecreasing, 
its limit is $1$ and an easy induction shows that $c_n \ge 1 - \frac{1}{n+1}$. 
Thus again by [ibid.], $c= c(\xi)$ for the sequence $\xi$ given by (\ref {e:reconstr c}), which then implies that $\frac{\xi}{\omega}$ is nondecreasing. 
And by the comment after Equation (\ref{e:concav}), it is also concave, hence $\xi =\eta_a$ for some $\eta\in\co*$. 
The remaining claims are also easy to verify directly.
\ep

Combining Proposition \ref {P:concav} and Corollary \ref {C: concav1} (wherein $c(\xi)_n \le 1$) we obtain

\bC{C: concav2} For $\xi\in\co*\setminus \Sigma(F)$ and $\xi_1=1$:
\item[(i)] $\xi =\eta_a$ for some $\eta\in\co*$ if and only if
\[
\frac {r(\xi)_n-\frac{1}{n}}{r(\xi)_{n-1}} + \frac 
{r(\xi)_n}{r(\xi)_{n+1}-\frac{1}{n+1}} \le 2 \quad (n > 1).
\]
\item[(ii)] $r(\eta_a)_{n+1} \le r(\eta_a)_n+ \frac{1}{n+1} < \big(1+ \frac{1}{n}\big)r(\eta_a)_n \quad (n \ge 1)$

\eC

>From Corollary \ref {C: concav2}(ii) one sees immediately that $r(\eta_a)_n \le H_n$ for all $n$. 
Of interest is a direct proof of this fact that avoids ratios of concavity considerations.
\begin{lemma}\label{L: H_n-a^2/a uparrow} Let  $0 \ne \eta \in \co*$.
\item[(i)]  $ r(\eta_a)_1 = H_1=1$ and if $\eta_2 > 0$, then $ r(\eta_a)_n < H_n$ for all $n>1$.
\item[(ii)] If $\eta_n > 0$ for all $n$, then $H_n - r(\eta_a)_n  $ 
is strictly increasing.
\item[(iii)] If $\eta$ is not summable, then $\lim (H_n - r(\eta_a)_n)= \infty$.
\end{lemma}

\begin{proof}
\item[(i)] $H_1=(\frac{\eta_{a^2}}{\eta_a})_1=1$.  For $n > 1$,
\begin{align*}
n(\eta_{a^2})_n &= \sum_{j=1}^{n} \frac{1}{j} \sum_{i=1}^{j} \eta_i
= \sum_{i=1}^{n} \sum_{j=i}^{n} \frac{\eta_i}{j}
= \eta_1 H_n + \sum_{i=2}^{n} \eta_i (H_n-H_{i-1}) \notag\\
&= H_n \sum_{i=1}^{n} \eta_i - \sum_{i=2}^{n} \eta_i H_{i-1} = H_n 
n(\eta_a)_n - \sum_{i=2}^{n} \eta_i H_{i-1}. \notag
\end{align*}

Thus
\be{e:H-ratio}
H_n - r(\eta_a)_n  = \frac{\sum_{i=2}^{n} \eta_i 
H_{i-1}}{\sum_{i=1}^{n} \eta_i} > 0
\ee
by the assumption that $\eta_2 >0$.
\item[(ii)] Using (\ref{e:H-ratio}),
\begin{align*} 
H_{n+1} &- (\frac{\eta_{a^2}}{\eta_a})_{n+1} -  \big(H_n - (\frac{\eta_{a^2}}{\eta_a})_n\big) = \frac{\sum_{i=2}^{n+1} \eta_i 
H_{i-1}}{\sum_{i=1}^{n+1} \eta_i} - \frac{\sum_{i=2}^{n} \eta_i H_{i-1}}{\sum_{i=1}^{n} \eta_i}\\
&= \eta_{n+1}\frac{H_n(\sum_{i=1}^{n} \eta_i) - \sum_{i=2}^{n} \eta_i H_{i-1}}{(\sum_{i=1}^{n+1} \eta_i)(\sum_{i=1}^{n} \eta_i)}
= \eta_{n+1}\frac{n(\eta_{a^2})_n }{(\sum_{i=1}^{n+1} \eta_i)(\sum_{i=1}^{n} \eta_i)} \\
&= \frac{(\frac{\eta_{a^2}}{\eta_a})_n}{(n+1)(\frac{\eta_{a}}{\eta})_{n+1}}>0,
\end{align*}
where the third equality was obtained in (i).
\item[(iii)]  Elementary from (\ref{e:H-ratio}) since $H_n \to \infty$.
\end{proof}

This lemma tells us that there are two extreme cases for the ratio of regularity for $\eta_a$:
when $r(\eta_a) \asymp 1$ (i.e., $\eta_a$ is regular) and when $r(\eta_a) \asymp \log$ (by which we mean more precisely: \linebreak
$\alpha \log n \le r(\eta_a)_n  \le \beta \log n$ for some $\alpha,~\beta > 0$ and all $n\ge 2$). 
As we will see in the next section, both cases play a special role for second order arithmetic mean cancellation. 

First we obtain some elementary comparisons between the ratios of 
regularity for $\eta$ and for $\eta_a$ evaluated at pairs of indices.

  \bL{L:ineq}
Let $\eta\in \co* \setminus \Sigma(F)$.
\item[(i)] $\frac{m}{n}  \left( (\eta_a)_m -\eta_n \right) + \eta_n \le (\eta_a)_n \le \frac{m}{n}\big((\eta_a)_m -\eta_m \big) + \eta_m$ \quad for all $n \ge m$.
\item[(ii)] $ \frac{m}{n}\big((\eta_{a^2})_m +(H_n-H_m) (\eta_a)_m\big) \le (\eta_{a^2})_n
\le \frac{m}{n}(\eta_{a^2})_m +(H_n-H_m) (\eta_a)_n$ \quad for all   $n \ge m$.
\item[(iii)]  Let  $n = [m(\frac{\eta_a}{\eta})_m]$ (integer part). 
Then $r(\eta_a)_n  >  \frac{1}{2}\log\,r(\eta)_m$.
  \eL
  \bp
\item[(i)] The inequalities  
are identities for $n=m$. 
For $n > m$,
\[
n (\eta_a)_n  =  \sum_{1}^{n} \eta_j = \sum_{1}^{m} \eta_j  + 
\sum_{m+1}^{n} \eta_j  \leq  m (\eta_a)_m  +  (n-m)\eta_m,
\]
and also
\[
n (\eta_a)_n  \ge m (\eta_a)_m  +  (n-m)\eta_n.
\]
\item[(ii)] The inequalities
are identities for $n=m$. For $n > m$,
\begin{align*}
n (\eta_{a^2})_n   &=   m (\eta_{a^2})_m  + \sum_{m+1}^{n} 
\frac{1}{j} ~j (\eta_a)_j   \\
&\leq   m(\eta_{a^2})_m  + \sum_{m+1}^{n}\frac{1}{j}~ n (\eta_a)_n   \\
&= m(\eta_{a^2})_m  +  (H_n - H_m)n (\eta_a)_n,
\end{align*}
and also
\begin{align*}
n (\eta_{a^2})_n   &\geq   m(\eta_{a^2})_m  + \sum_{m+1}^{n}\frac{1}{j}~ 
m (\eta_a)_m   \\
&= m(\eta_{a^2})_m  + (H_n - H_m) m (\eta_a)_m.
\end{align*}
\item[(iii)] 
By the definition of $n$, 
$\frac{n+1}{m} > (\frac{\eta_a}{\eta})_m\ge \frac{n}{m}$. 
Hence by (i), (ii), (\ref{e: Hn-Hm}), and the trivial inequality $(\eta_{a^2})_m \ge (\eta_{a})_m$,
\begin{align*}
(\frac{\eta_{a^2}}{\eta_a})_n &> \frac{m \left((\eta_{a^2})_m + (\eta_a)_m \log \frac{n+1}{m+1}\right)}{m (\eta_a)_m + (n-m)\eta_m} \\
&> \frac{ (\eta_a)_m\left( 1+  \log \frac{n+1}{m+1}\right)}{(\eta_a)_m + \frac{n}{m}\eta_m}\\
&> 		\frac {\log \frac{n+1}{m}}{2}\\
&>  \frac{1}{2} \log (\frac{\eta_a}{\eta})_m.
\end{align*}

  \ep

\noindent This proof evolved from work of K. Davidson with the second named author.
The original proof was obtained by the authors using ratios of 
concavity.\medskip

Lemma \ref{L:ineq} provides a direct and quantitative proof for the following result (proven implicitly in \cite[Theorem IRR]{jV89} and explicitly in 
\cite[Theorem 3.10]{DFWW}; see also [ibid., Remark 3.11]).

\begin{corollary}\label{C: g regular iff a(g)}
For $\eta \in \text{c}_{\text{o}}^*$, $\eta$ is regular if and only if $\eta_a$ is regular.
\end{corollary}

Now we consider the second ``extreme" case, namely when $r(\eta_a) \asymp \log$, meaning in the sense that, except for $n=1$, 
$\alpha \log n \le r(\eta_a)_n \le \beta \log n$ for some $\alpha,~\beta > 0$.

\bP{P: exp Delta2 equiv log asymp a2/a} Let $0 \ne \eta \in \co*$. 
Then $r(\eta_a) \asymp \log$ if and only if $\underset{m}{\sup}\,\frac{m^2(\eta_a)_{m^2}}{m(\eta_a)_m} < \infty$.
\eP
\bp
Assume $\gamma := \underset{m} 
\sup\,\frac{m^2(\eta_a)_{m^2}}{m(\eta_a)_m}  <  \infty$.
For all $n \geq m$, it follows from Lemma \ref{L:ineq}(ii) that
\[
(\eta_{a^2})_n  >  \frac{m}{n}(H_n - H_m)(\eta_a)_m 
\]
and hence by (\ref{e: Hn-Hm})
\begin{equation*}
\big(\frac{\eta_{a^2}}{\eta_a}\big)_n  >  \frac{m}{n} \frac{(\eta_a)_m}{(\eta_a)_n}~(H_n - H_m)
>  \frac{m}{n} \frac{(\eta_a)_m}{(\eta_a)_n}\log\frac{n+1}{m+1}
\end{equation*}
Set $m = [\sqrt n]+1$.
Then
\begin{align*}
\left(\frac{\eta_{a^2}}{\eta_a}\right)_n  &> 
\frac{m(\eta_a)_m}{m^2(\eta_a)_{m^2}}\, 
\frac{m^2(\eta_a)_{m^2}}{n(\eta_a)_n}\log\frac{n+1}{\sqrt n+1} 
	\qquad &\\
&\ge 
\frac{1}{\gamma}\frac{m^2(\eta_a)_{m^2}}{n(\eta_a)_n}\log\frac{n+1}{\sqrt n+1}	\qquad &\text{by the definition of $\gamma$ }\\
&\ge 	\frac{1}{\gamma}\log\frac{n+1}{\sqrt n+1}	&\text{by the monotonicity of $k(\eta_a)_k$ since $m^2 > n$} \\
&\ge  \frac{1}{3\gamma}\log n &\text {for $n$ sufficiently large.}
\end{align*}
On the other hand, by Lemma \ref{L: H_n-a^2/a uparrow}(i) and (\ref{e: Hn-Hm}),
$r(\eta_a)_n < H_n < \log n +1$, so $r(\eta_a) \asymp \log$.

Conversely, assume $r(\eta_a) \asymp \log$, i.e.,  $\alpha \log n \le r(\eta_a)_n  \le \beta \log n$ for some $\alpha,~\beta > 0$ and all $n\ge 2$. 
Then
\[
\frac{m^2(\eta_a)_{m^2}}{m(\eta_a)_m} = 
\frac{\big(\frac{\eta_{a^2}}{\eta_a}\big)_m}{\big(\frac{\eta_{a^2}}{\eta_a}\big)_{m^2}}\frac{m^2(\eta_{a^2})_{m^2}}{m(\eta_{a^2})_m}
\le\frac{\beta}{\alpha}\frac{\log m}{\log m^2}\frac{m^2(\eta_{a^2})_{m^2}}{m(\eta_{a^2})_m}
=\frac{\beta}{2\alpha}\frac{m^2(\eta_{a^2})_{m^2}}{m(\eta_{a^2})_m}
\]
and similarly,
\[
\frac{m^2(\eta_a)_{m^2}}{m(\eta_a)_m}\ge 
\frac{\alpha}{2\beta}\frac{m^2(\eta_{a^2})_{m^2}}{m(\eta_{a^2})_m},
\]
i.e.,
\be{e:higher order}
\frac{m^2(\eta_a)_{m^2}}{m(\eta_a)_m} \asymp 
\frac{m^2(\eta_{a^2})_{m^2}}{m(\eta_{a^2})_m}.
\ee
Now, from Lemma \ref {L:ratio}(ii) applied to $\xi = \eta_a$, $\phi 
= \alpha \log$, and $n = m^2$, we have
\[
  \frac{m^2(\eta_{a^2})_{m^2}}{m(\eta_{a^2})_m}
\le m \big(\frac{m}{m^2}\big)^{1-\frac{2}{\alpha \log m}}= 
e^{\frac{2}{\alpha}}.
\]
\ep

\noindent By analogy to the $\Delta_2$-condition for monotone sequences ($\underset{m}\sup\frac{\phi(2m)}{\phi(m)}<\infty$), we say 
that $\eta$ satisfies the \textit{exponential $\Delta_2$-condition} if, $\underset{m}\sup \frac{m^2(\eta_a)_{m^2}}{m(\eta_a)_m} < \infty$. 
i.e., $\sum_{i=1}^{m^2} \eta_i \le c\sum_{i=1}^{m} \eta_i $ for some $c > 0$ and all $m$.  \\

An obvious consequence of (\ref{e:higher order}) is:

\bC{C:higher order}
Let $0 \ne \eta \in \co*$. If $\eta$ satisfies the exponential $\Delta_2$-condition then  $\eta_{a^p}$ satisfies the exponential 
$\Delta_2$-condition for every $p \in \mathbb N$.
\eC

\begin{example}\label{R: log^pn}
Among the  $\co*$-sequences that satisfy the exponential $\Delta_2$-condition are all  summable sequences, all sequences 
$<\frac{\log^pn}{n}>$,  $<\frac{\log^pn(\log\log n)^q}{n}>$, etc. (starting from wherever they become monotone decreasing).
Among sequences that do not satisfy the exponential $\Delta_2$-condition are all regular sequences, the sequences in 
Examples \ref{E: Example 6} and \ref{e: omega1/2}. Likewise for the sequence (starting from wherever it becomes monotone decreasing)
\[
\eta := \left< \frac{e^{\int_{e^e}^n \frac{1}{t\log\log\, t}}}{n\log^2\log\, n}\right> \quad \text{for which \quad $r(\eta_a)\asymp \log \log n$.}
\]
We skip the work to verify the stated properties for these sequences.

This last example shows that the condition that $r(\eta_a)$ be equivalent to a monotone sequence (see Theorem \ref{T: monotone}) is more general than that $r(\eta_a)$ be equivalent to $log$ or $1$ (see Theorems \ref{T: hat or log} and \ref{T: equality cancel}). 
  \end{example}

Both regular sequences and sequences that satisfy the exponential $\Delta_2$-condition are special cases of sequences $\eta$ for which 
$r(\eta_a)$ is equivalent to a monotone sequence. 
In that case, since $r(\eta_a) \ge 1$, $r(\eta_a)$ is either equivalent to a sequence increasing to infinity or to a constant sequence (when $\eta$ is regular); in either case, it is equivalent to a nondecreasing sequence so that Lemma \ref{L:ratio}(ii) applies. \\

\begin{lemma}\label{L: monotone}
If  $0 \ne \eta \in \text{c}_{\text{o}}^*$ and $r(\eta_a)$ is equivalent to a monotone sequence, 
then there are constants $K, M >0$  for which, if $m \ge M$, then $(\eta_{a^2})_n \le K\frac{m}{n}\log\frac{n}{m}\, (\eta_a)_m$ for some $n > m$.
\end{lemma}

\begin{proof}
Assume without loss of generality that $r(\eta_a)$ is equivalent to a monotone nondecreasing sequence $\phi$ and assume for simplicity's sake that $ \phi \le r(\eta_a)\le \beta\phi$ for some $\beta >0$. 
Let $M\in \mathbb N$ be an integer for which $M\phi_M > 2$. Then for all $n \ge m\ge M$, by Lemma \ref{L:ratio}(ii) applied to $\eta_a$,
\[
(\eta_{a^2})_n \le \big(\frac{m}{n}\big)^{1- \frac{2}{\phi_m}}(\eta_{a^2})_m\
\le \beta \phi_m \big(\frac{m}{n}\big)^{1- \frac{2}{\phi_m}}(\eta_a)_m
= \beta \phi_m \frac{m}{n}\,e^{\frac{2 \log \frac{n}{m}}{\phi_m}}(\eta_a)_m.
\]
Choosing $n:=[me^{\phi_m}]$ one has $\frac{n}{m}\le e^{\phi_m} < \frac{n+1}{m}$, hence $\phi_m < \log \frac {n+1}{m} < 2\log\frac {n}{m} \le 2\phi_m$. 
Thus
\[
(\eta_{a^2})_n \le 2\beta e^2\, \frac{m}{n} \log\frac {n}{m}\,(\eta_a)_m.
\]
\end{proof}

\section{\leftline{\bf Arithmetic mean cancellations of second 
order}} \label{S: 4}

First order lower arithmetic mean cancellation characterizes am-closed ideals, i.e., ideals $I$ for which $I= I^-:=\,_a( I_a)$.
Indeed for a fixed ideal $I$, 

\[ J_a \subset I_a \Rightarrow J \subset I \quad \text{if and only if}\quad I= \,_a(I_a) \quad \text{(\cite[Lemma 6.1(C)]{vKgW04-Density}).} \]

The second order analog of this property involves second order am-closure, $_{a^2}( I_{a^2})$.

\begin{proposition}\label{P: closures} For a fixed ideal $I$, 
$J_{a^2} \subset I_{a^2} \Rightarrow J_a \subset I_a$  if and only if $I^- = \,_{a^2}( I_{a^2})$.
\end{proposition}
\begin{proof}
The condition is necessary. Indeed, from the general identity $I_{a^2} = (_{a^2}(I_{a^2}))_{a^2}$ (see Section \ref{S: 2}), 
it follows from the hypothesis that $ (_{a^2}(I_{a^2}))_a \subset I_a$ and hence
\[
I^- = \,_a(I_a) \supset \, _a((_{a^2}(I_{a^2}))_a) = \, _a((_a[_a(I_{a^2})])_a) = \,_a[_a(I_{a^2})] = \,_{a^2}( I_{a^2}) \supset I^-
\]
where the third equality holds from the general identity $\,_a((_aL)_a) = \,_aL\,$  and the last inclusion also holds for any ideal.
Conversely, if $I^- = \,_{a^2}( I_{a^2})$ and  $J_{a^2} \subset I_{a^2}$, then $J\subset \,_{a^2}(J_{a^2})\subset \,_{a^2}(I_{a^2})$ and hence
  \[
J_a \subset (_{a^2}(I_{a^2}))_a = (I^-)_a = (_a(I_a))_a = I_a.
\]
\ep

\bR{R: 2nd clos}
\item[(i)] The last step in the above proof shows that if $I^- = \,_{a^2}( I_{a^2})$, then $I_a = (I_a)^{-o}$. 
It is easy to see that the converse also holds, i.e., that $I^- = \,_{a^2}( I_{a^2})$ if and only if  $I_a = (I_a)^{-o}$.
\item[(ii)]  
Since $I\subset I^-\subset \,_{a^2}(I_{a^2})$ for arbitrary ideals $I$, if $I=\,_{a^2}(I_{a^2})$ then $I^-=\,_{a^2}(I_{a^2})$.
The converse implication fails in general. 
For instance, if $L$ is the countably generated ideal provided by \linebreak
\cite[Example 5.5]{vKgW04-Density} for which $L\subsetneq L_a = L_{a^2}$, 
then $_{a^2}(L_{a^2})=L_a$ and hence $L \ne \,_{a^2}(L_{a^2})$ but $L^- = L_a =\, _{a^2}(L_{a^2})$.
Furthermore, the converse implication can fail even when $I$ is principal.
Indeed, if $I=(\eta)$ is not am-stable, then by \cite[Theorem 2.11]{vKgW04-Soft}, it is not am-closed, hence $I \ne \,\,_{a^2}(I_{a^2})$ 
(recall that $_aJ$ is am-closed for every ideal $J$).
However by Theorem \ref{T: monotone}(i) below, if $r(\eta_a)$ is equivalent to a monotone sequence (e.g., $\eta = \omega$ hence $r(\eta_a) \asymp \log$), 
then $I^-=\,_{a^2}(I_{a^2})$ by Proposition \ref{P: closures}. 

\eR

The condition  $I^- = \,_{a^2}( I_{a^2})$ is not very ``transparent", 
not even for principal ideals.  A natural question is whether this 
condition might be automatically satisfied for all ideals. 
As is easy to see (cfr. proof of Theorem \ref {T: monotone} below), 
if second order inclusion cancellation were to hold for all pairs of principal 
ideals it would  hold also for all pairs of general ideals.
Furthermore, if equality cancellation were to hold for all pairs of 
principal ideals, inclusion cancellation would then  also hold for 
all pairs of principal ideals. 
Indeed, given two principal ideals $I$ and $J$ 
and setting $L:=I+J$, we see that  $J_{a^2} \subset I_{a^2}$ is equivalent to $L_{a^2} = I_{a^2}$ which would then imply $L_a=I_a$ and hence $I_a\supset J_a$.  \\

It is trivial to see that first order equality cancellation 
does not hold for all pairs of principal ideals, e.g., because all 
nonzero principal ideals contained in the trace class have the same 
arithmetic mean ideal $(\omega)$.

Whether second order equality cancellation holds for all pairs of 
principal ideals or not is indeed a reformulation of a question asked by M. Wodzicki. 
The following example answers this question in the negative.\\

We need first the following identities for ``step sequences." \\

Let $\zeta\in \co*$ be a step sequence based on a strictly 
increasing sequence of indices $m_k$ starting with $m_0=0$,
i.e., $\zeta_j= \epsilon_k $ for $m_k< j \le m_{k+1} $ for some 
strictly decreasing sequence $\epsilon_k \to 0$.
We will need the following formulas for the sequences $\zeta_a$ and 
$\zeta_{a^2}$:
\[
j(\zeta_a)_j= \sum_1^j\zeta_j=
\begin{cases}
j\epsilon_o \quad\quad\quad & \text{ for $0< j \le m_1 $}\\
m_k(\zeta_a)_{m_k} +(j-m_k)\epsilon_k & \text{ for $m_k< j \le 
m_{k+1} $ and $k\ge 1$}
\end{cases}
\]
hence
\begin{equation}\label{xia}
(\zeta_a)_j=
\begin{cases}
\epsilon_o \quad\quad\quad & \text{ for $0< j \le m_1 $}\\
\frac{m_k}{j}((\zeta_a)_{m_k} -\epsilon_k) + \epsilon_k &\text{ for 
$m_k< j \le m_{k+1}$ and $k\ge 1.$}
\end{cases}
\end{equation}
Consequently, for $k>0$ and $m_k< j \le m_{k+1}$,
\begin{align*}
j(\zeta_{a^2})_j&= m_k(\zeta_{a^2})_{m_k} + \sum_{i=m_k+1}^j  \big( 
\frac{m_k}{i}((\zeta_a)_{m_k} -\epsilon_k )+\epsilon_k\big) \\
&= m_k(\zeta_{a^2})_{m_k} +  m_k\big((\zeta_a)_{m_k} - 
\epsilon_k)(H_j-H_{m_k}\big) + (j-m_k)\epsilon_k,
\end{align*}
and for $k=0$ and $0< j \le m_1$, $ j(\zeta_{a^2})_j=j\epsilon_0$. Therefore
\begin{equation}\label{xiaa}
(\zeta_{a^2})_j =
\begin{cases}
\epsilon_o \quad\quad\quad & \text{ for $0< j \le m_1 $}\\
\frac{m_k}{j}\Big((\zeta_{a^2})_{m_k} - \epsilon_k +((\zeta_a)_{m_k} 
- \epsilon_k)(H_j-H_{m_k})\Big) +\epsilon_k &\text{ for $m_k< j \le 
m_{k+1}$. }
\end{cases}
\end{equation}

\

\begin{example}\label{E: Example 6}
Principal ideals $J \subset I $ for which $J_{a^2}= I_{a^2}$ but $J_a\ne I_a$.\\
\end{example}
\begin{proof}[Construction]
To define the principal ideals it suffices to provide sequences $\xi$ and $\eta$ in $\co*$ to generate them with the properties that $\xi \le \eta$ and  
$\xi_{a^2} \asymp \eta_{a^2}$ but for which $\eta_a \ne O(\xi_a)$ (see Section \ref{S: 2} on principal ideals). 
Construct inductively an increasing sequence of positive integers $m_k$ with $m_1=1$,
then define the sequence of indices $n_k := [e^{k^2}m_k]$ and  the 
sequence $\delta_k$ defined recursively from $\delta_1=1$ and 
$\delta_{k+1}= e^{-k^2}\delta_k$, i.e.,  $\delta_k 
=e^{-\sum_{p=1}^{k-1}p^2}$ for $k>1$. Now, using the sequences $m_k$, 
$n_k$, and $\delta_k$,  define the two monotone sequences $\xi$ and 
$\eta$  setting $\xi_1=\eta_1 = 1$ and for every $k\ge 1$,
\[
\xi_j:= e^{-k^2}\delta_k ~\text{ for }~ m_k < j \le m_{k+1}
\]
and
\[
\eta_j :=
\begin{cases}
ke^{-k^2}\delta_k 	&\text {for  $m_k < j \le n_k$} \\
e^{-k^2}\delta_k	&\text{for $n_k < j \le  m_{k+1}.$}
\end{cases}
\]
Notice that   $ \eta_{m_k} = \xi_{m_k}=\delta_k$ for every $k$. 
Clearly, $\xi \le \eta$ and both sequences are in  $\co*$.
Assume the construction of the sequence $m_k$ up to $k\ge 1$ and 
choose $m_{k+1}> n_k$ and sufficiently large to insure that $(\eta_{a^2})_{m_{k+1}} \le 
(1+ \frac{1}{k})e^{-k^2}\delta_k$, which can be achieved using 
Equation (\ref{xiaa}).
As a consequence,
\begin{equation}\label{clock}
\delta_k =  \eta_{m_k} = \xi_{m_k}\le (\xi_a)_{m_k} \le 
(\eta_a)_{m_k} \le (\eta_{a^2})_{m_k}  \le (1+ \frac{1}{k})\delta_k 
~\text{~and~}~
\end{equation}
\begin{equation*}
\delta_k \le (\xi_a)_{m_k} \le(\xi_{a^2})_{m_k}  \le (1+ \frac{1}{k})\delta_k
\end{equation*}
  and hence
  \begin{equation}\label{asympt}
(\xi_a)_{m_k}  \sim (\eta_a)_{m_k}
\sim (\xi_{a^2})_{m_k}  \sim 
(\eta_{a^2})_{m_k}\sim \delta_k.
  \end{equation}
 From (\ref{xia}), (\ref{asympt}) and
\begin{equation}\label{e:n/m}
\frac{m_k}{n_k} = e^{-k^2}\frac {e^{k^2}m_k}{[e^{k^2}m_k]} \sim e^{-k^2}
\end{equation}
we have
\begin{equation}\label{e:etaan}
\frac{(\xi_a)_{n_k}}{\delta_k} = \frac{m_k}{n_k} 
\big(\frac{(\xi_a)_{m_k}}{\delta_k} -  e^{-k^2}\big)+e^{-k^2}  \sim 2 
e^{-k^2}
\end{equation}
and
\begin{equation}\label{e:xian}
\frac{(\eta_a)_{n_k}}{\delta_k} = \frac{m_k}{n_k} 
\big(\frac{(\xi_a)_{m_k}}{\delta_k} -  ke^{-k^2}\big)+ke^{-k^2}  \sim 
k e^{-k^2}.
\end{equation}
As a consequence, $\eta_a \ne O(\xi_a)$, and hence $(\xi)_a\ne (\eta)_a$.

 From (\ref{xiaa}), (\ref{asympt}) and (\ref{e:n/m}) we have
\be{e: asym xi}
\begin{aligned}
\frac{(\xi_{a^2})_{n_k}}{\delta_k} &= \frac{m_k}{n_k} 
\big\{\frac{(\xi_{a^2})_{m_k}}{\delta_k} -  e^{-k^2}
+ (\frac{(\xi_{a})_{m_k}}{\delta_k} -  e^{-k^2} )(H_{n_k}-H_{m_k}) 
\big\}+e^{-k^2}\\
&\sim    e^{-k^2}(1 + \log\frac{n_k}{m_k} ) +e^{-k^2} \sim k^2e^{-k^2}
\end{aligned}
\ee
and
\be{e: asym eta}
\begin{aligned}
\frac{(\eta_{a^2})_{n_k}}{\delta_k} &= \frac{m_k}{n_k} 
\big\{\frac{(\eta_{a^2})_{m_k}}{\delta_k} -  ke^{-k^2}
+ (\frac{(\eta_{a})_{m_k}}{\delta_k} -  ke^{-k^2} )(H_{n_k}-H_{m_k}) 
\big\}+ke^{-k^2}\\
&\sim    e^{-k^2}(1 + \log\frac{n_k}{m_k} ) +ke^{-k^2} \sim k^2e^{-k^2}
\end{aligned}
\ee
As an aside relating to the sufficient condition in Theorem \ref{T: monotone} below we note the equalities 
$\limsup \frac{\xi_{a^2}}{\xi_a} =\limsup \frac{\eta_{a^2}}{\eta_a} = \infty$, while from (\ref{asympt}), 
$\liminf \frac{\xi_{a^2}}{\xi_a} =\liminf  \frac{\eta_{a^2}}{\eta_a} = 1$, 
and hence neither $\frac{\xi_{a^2}}{\xi_a}$ nor $\frac{\eta_{a^2}}{\eta_a}$ are equivalent to a monotone sequence.
The conclusions of this example together with Theorem \ref{T: monotone} of course also implies this.

Now we give the crux of the proof, that $(\eta_{a^2})_j \le 2 (\xi_{a^2})_j$ for $j$ sufficiently large, 
implying that $(\xi)_{a^2}= (\eta)_{a^2}$.
If $m_k < j \le n_k$, from (\ref{xiaa})-({\ref{clock}}) we have for all $k\ge2$, 

\begin{align*}
&\frac{2(\xi_{a^2})_j  - (\eta_{a^2})_j }{\delta_k}= \\
&\frac{m_k}{j}\Big\{\frac{2(\xi_{a^2})_{m_k}-(\eta_{a^2})_{m_k}}{\delta_k} +(k-2)e^{-k^2}
+\big(\frac{2(\xi_a)_{m_k}-(\eta_a)_{m_k}}{\delta_k} + (k-2)e^{-k^2}\big)(H_j-H_{m_k})\Big\}\\
&+(2-k)e^{-k^2}\\
&\ge 
\frac{m_k}{j}\Big\{\frac{2(\xi_{a^2})_{m_k}-(\eta_{a^2})_{m_k}}{\delta_k} 
+
\frac{2(\xi_a)_{m_k}-(\eta_a)_{m_k}}{\delta_k} (H_j-H_{m_k}) 
-ke^{-k^2}\frac{j}{m_k}\Big\}\\
&\ge \frac{m_k}{j}\Big\{1- \frac{1}{k} + (1- 
\frac{1}{k})(\log\frac{j}{m_k} -  \frac{1}{m_k+1}\ 
)-ke^{-k^2}\frac{j}{m_k}\Big\}\\
&\ge \frac{m_k}{j}\Big\{\frac{1}{2}(\frac{1}{2}+ \log\frac{j}{m_k}) 
-ke^{-k^2}\frac{j}{m_k}\Big\},
\end{align*}
where the second last inequality follows from  (\ref{clock}) and (\ref{e: Hn-Hm}).
Notice that $1 < \frac{j}{m_k}\le \frac{n_k}{m_k} \le e^{k^2}$ and then
elementary calculus shows that the function $\phi(x):= \frac{1}{4} + \frac{1}{2} \log\,x -ke^{-k^2}x$ attains its absolute minimum on the 
interval $[1, e^{k^2}]$ for $x=1$ and $\phi(1) = \frac{1}{4} - ke^{-k^2} >0$.
Thus $2(\xi_{a^2})_j  > (\eta_{a^2})_j$ for all $m_k < j \le n_k$.
For $n_k < j \le m_{k+1}$, again from (\ref{xiaa}) we have
\be{e:ineq}
\begin{aligned}
&\frac{2(\xi_{a^2})_j  - (\eta_{a^2})_j }{\delta_k}= \\
&=\frac{n_k}{j}\Big (\frac{2(\xi_{a^2})_{n_k}-(\eta_{a^2})_{n_k}}{\delta_k} 
-e^{-k^2}+(\frac{2(\xi_a)_{n_k}-(\eta_a)_{n_k}}{\delta_k} 
-e^{-k^2})(H_j-H_{n_k})\Big) +e^{-k^2}.
\end{aligned}
\ee
Since $\frac{2(\xi_{a^2})_{n_k}  - (\eta_{a^2})_{n_k} }{\delta_k} \sim k^2e^{-k^2}$ by (\ref{e: asym xi})-(\ref{e: asym eta})
and $\frac{2(\xi_a)_{n_k}  - (\eta_a)_{n_k} }{\delta_k} \sim -ke^{-k^2}$ by (\ref{e:etaan})-(\ref{e:xian}), for $k$ sufficiently large we have
$\frac{2(\xi_{a^2})_{n_k} -(\eta_{a^2})_{n_k} }{\delta_k} \ge \frac{1}{2} k^2e^{-k^2}$ and $\frac{2(\xi_a)_{n_k} - (\eta_a)_{n_k} }{\delta_k} \ge -2ke^{-k^2}.$ 
Thus from (\ref{e:ineq}) and (\ref{e: Hn-Hm}), for $k$ sufficiently large,

\begin{align*}
\frac{2(\xi_{a^2})_j  - (\eta_{a^2})_j }{e^{-k^2}\delta_k}&\ge
\frac{n_k}{j}\Big\{\frac{1}{2}k^2 -1 + ( -2k -1)(H_j- H_{n_k}) 
+\frac{j}{n_k} \Big\}\\
&\ge
\frac{n_k}{j}\Big\{\frac{1}{3}k^2  -3k \log\,\frac{j}{n_k} 
+\frac{j}{n_k} \Big\}.
\end{align*}

Again, elementary Calculus shows that the function $\psi(x):= \frac{1}{3}k^2  -3k \log\, x +x$
attains its absolute minimum  on the interval $[1, \infty)$ for $x=3k$ and $\psi(3k) > 0$ when $k$ is sufficiently large. 
Thus for large $k$, $2(\xi_{a^2})_j \ge (\eta_{a^2})_j$ also for all $n_k < j \le m_{k+1}$, which completes the proof.\\
\end{proof}

\bR{R: Example 6} The construction in Example \ref {E: Example 6} illustrates some features of the behavior of  ``step sequences," 
their first order and second order arithmetic means. On a long interval of constancy of a sequence $\zeta$, both the first and second order 
means $\zeta_a$ and $\zeta_{a^2}$ approach the value of the sequence $\zeta$, with $\zeta_a$ approaching it faster than $\zeta_{a^2}$. 
This ``resetting the clock" greatly simplifies the computations.

Following a step down, $\zeta_{a^2}$ decreases much slower than  $\zeta_a$. 
Thus in the construction of the example, by having different step sizes for  $\eta$ and $\xi$ it would have been relatively straightforward 
to achieve large ratios for $\frac{\eta_a}{\xi_a}$. 
The delicate point was to do so while simultaneously keeping the difference $2\xi_{a^2}- \eta_{a^2}$ positive not only on the first part of the interval but on the second one as well, where the two sequences are equal (which was not automatic due to the delay in the decrease of  $\eta_{a^2}$).
\eR

A consequence of Example  \ref{E: Example 6} is that not all ideals, and not even all principal ideals, satisfy the necessary and sufficient 
condition $I^- = \,_{a^2}( I_{a^2})$ for the cancellation  $J_{a^2} \subset I_{a^2} \Rightarrow J_a \subset I_a$ to hold (Proposition \ref {P: closures}). However we are still left wanting a more usable conditions. 
Clearly, am-stability is trivially  sufficient even for general ideals, since in this case $I= I_a =\,_aI$ and hence $I=I^- = \,_{a^2}( I_{a^2})$. 
It is, however, far from necessary. 
Indeed a much more general sufficient condition is provided by the following theorem.

\bT{T: monotone}
\item[(i)] Let $I = (\eta)$ be principal ideal, let $r(\eta_a)$ be equivalent to a monotone sequence and let $J$ be an arbitrary ideal. 
Then $J_{a^2} \subset I_{a^2} \Rightarrow J_a \subset I_a$.
\item[(ii)] A sufficient condition on a general ideal $I$ for the second order am-inclusion cancellation implication in (i) to hold for 
arbitrary ideals $J$, is that every $\mu \in \Sigma(I)$ is dominated by some $\eta \in \Sigma(I)$ (i.e., $\mu \le \eta$) for which 
$r(\eta_a)$ is equivalent to a monotone  sequence.
\eT
\begin{proof}
\item[(i)]
It suffices to prove the cancellation property for the case that $J$ itself is principal. 
Indeed, if $J$ is a general ideal with $J_{a^2} \subset I_{a^2}$ and if $\rho \in \Sigma (J_a)$, i.e., $\rho \le \xi_a$ for some $\xi\in\Sigma(J)$, 
then $(\xi)_{a^2} \subset J_{a^2} \subset I_{a^2}$. 
We claim that $(\xi_a) \subset I_a$, whence $(\rho) \subset I_a$ and hence $J_a \subset I_a$ by the arbitrariness of $\rho$.

Since $I = (\eta)$ is principal, so are $I_a$ and $I_{a^2}$, indeed $I_a= (\eta_a)$ and $I_{a^2}= (\eta_{a^2})$. 
Thus the inclusions $(\xi_a) \subset I_a$ (resp., $(\xi)_{a^2} \subset I_{a^2}$) are equivalent to the conditions  $\xi_a = O( \eta_a)$ 
(resp., $\xi_{a^2} = O( \eta_{a^2})$).  
Thus, to prove the claim, it suffices to prove that if $\frac{\xi_a}{\eta_a}$ is unbounded, then so is $\frac{\xi_{a^2}}{\eta_{a^2}}$.  
By Lemma \ref{L: monotone}, there are constants $K, M >0$ for which if, $m \ge M$, then
\[
(\eta_{a^2})_{n_m} \le K\frac{m}{{n_m}}\log\frac{n_m}{m} (\eta_a)_m \quad \text{ for some $n_m > m$}.
\]
  By Lemma \ref {L:ineq}(ii) and (\ref{e: Hn-Hm}),
\[
(\xi_{a^2})_{n_m} \ge \frac{m}{n_m}(H_{n_m}-H_m)(\xi_a)_m > \frac{m}{n_m}\log\frac {n_m+1}{m+1}(\xi_a)_m \ge \frac{1}{2}\frac{m}{n_m}\log\frac {n_m}{m}(\xi_a)_m.
\]
Thus
\[
\Big(\frac{\xi_{a^2}}{\eta_{a^2}}\Big)_{n_m} \ge \frac{1}{2K} \Big(\frac{\xi_a}{\eta_a}\Big)_m.
\]
Hence the unboundedness of $\frac{\xi_a}{\eta_a}$ implies the unboundedness of $\frac{\xi_{a^2}}{\eta_{a^2}}$.
\item[(ii)]  Assume that  $J_{a^2} \subset I_{a^2}$ and that  $\rho \in \Sigma (J_a)$, i.e., $\rho \le \xi_a$ for some $\xi\in\Sigma(J)$. 
Then $\xi_{a^2}\in\Sigma(I_{a^2})$, i.e.,  $\xi_{a^2} = O(\mu_{a^2} )$ for some $\mu \in \Sigma(I)$, 
hence $\xi_{a^2} = O(\eta_{a^2} )$ for some $\eta\in \Sigma(I)$ for which $r(\eta_a)$ is equivalent to a monotone sequence. 
By (i), the inclusion $( \xi_{a^2} )\subset (\eta_{a^2} )$ implies the inclusions $(\rho)\subset ( \xi_a)\subset (\eta_a)\subset I_a$. 
By the arbitrariness of $\rho$, we conclude that $J_a\subset I_a$.
\end{proof}

We do not know if the condition in part (i), that $r(\eta_a)$ is equivalent to a monotone sequence, 
is necessary for the inclusion cancellation in Theorem \ref{T: monotone}. 
However, we see from Example \ref{e: omega1/2} below that it is not sufficient for the second order reverse inclusion cancellation. 
Sufficient conditions for that cancellation to hold are given by the following theorem.

\bT{T: hat or log}
\item[(i)] Let $I = (\eta)$ be a principal ideal, let $r(\eta_a)\asymp \log$ or $\eta \asymp \widehat \eta$ (see end of Section \ref{S: 2}), and let $J$ be an arbitrary ideal. 
Then $J_{a^2} \supset I_{a^2} \Rightarrow J_a \supset I_a$.
\item[(ii)] A sufficient condition on a general ideal $I$ for the second order inclusion am-cancellation implication in (i) to hold for arbitrary ideals $J$ is that every $\mu \in \Sigma(I)$ is dominated by some $\eta \in \Sigma(I)$ for which $r(\eta_a) \asymp \log$ or $\eta \asymp \widehat \eta$.
\eT

\begin{proof}
\item[(i)] As in the proof of Theorem \ref{T: monotone}(i), we reduce the proof to the case where $J$ also is principal. 
Indeed, $I = (\eta)$ and $I_{a^2}\subset J_{a^2}$ implies that $\eta_{a^2} \le \xi_{a^2} $ for some $\xi \in \Sigma(J)$. 
In the case when $r(\eta_a)\asymp \log$, by Lemma \ref {L: H_n-a^2/a uparrow}(i) and (\ref{e: Hn-Hm}),
\[
(\eta_a)_n\le K \frac{(\eta_{a^2})_n}{\log n} \le K \frac{(\xi_{a^2})_n}{\log n} =K\frac{r(\xi_a)}{\log n}(\xi_a )_n \le K\frac{H_n}{\log n} (\xi_a )_n
\le 2K (\xi_a )_n
\]
for some $K>0$ and all $n > 2$.
In the case when $\eta \asymp \widehat \eta $, by \cite[Proposition 6.17]{vKgW04-Density}, $\eta$ is regular.
Then by [ibid., Definition 6.10 and Corollary 6.15], $(\eta_{a^2}) \subset (\xi_{a^2})$ implies $\widehat{(\eta_a)} \subset (\xi_a)$ and so $(\eta_a) = (\eta) = (\widehat{\eta}) = \widehat{(\eta)} = \widehat{(\eta_a)} \subset (\xi_a)$.
In either case, this shows that $I_a=(\eta_a)\subset (\xi_a)\subset J_a$.
\item[(ii)]  For every $\mu\in\Sigma(I_a)$, there is a $\xi \in \Sigma(I)$ for which $\mu \le \xi_a$. 
By the hypothesis, there is also an $\eta \in \Sigma(I)$ with $\xi \le \eta_a$ for which $r(\eta_a) \asymp \log$ or $\eta \asymp \widehat \eta$. 
Since $(\eta)_{a^2}\subset I_{a^2} \subset J_{a^2}$, by (i) we obtain $(\mu)\subset (\xi_a) \subset (\eta)_a \subset J_a$. 
Since $\mu$ is arbitrary, we conclude that $I_a \subset J_a$.
\ep

As mentioned earlier, one sees from the last example listed in Example \ref {E: Example 6} that the condition that $r(\eta_a)$ be equivalent to a monotone  sequence (as in Theorem \ref{T: monotone}) is more general than that $r(\eta_a)$ be equivalent to $log$ or $1$ (as in Theorems \ref{T: hat or log} and \ref{T: equality cancel}). 
\\

Notice that neither the condition that $r(\eta_a)$ is equivalent to a monotone sequence, nor the more restrictive condition that 
$r(\eta_a)$ is bounded, i.e.,  that $\eta$ is regular, are sufficient for the cancellation in part (i) to hold.
Recall that $(\omega^p)$ is am-stable for all $0<p<1$ but $(\omega^p)\ne \widehat{(\omega^p)}= (\omega^{p'})$ for $\frac{1}{p}- \frac{1}{p'} = 1$ by 
\cite[Corollary 6.16]{vKgW04-Density}. 
Thus we see that \[(\omega^{1/2})=(\omega^{1/2})_a \subset (\xi)_a \not\Rightarrow (\omega^{1/2}) \subset (\xi).\] 
The analogous cancellation for second order ideals is a priori different, although the same conclusion holds.

\begin{example}\label{e: omega1/2}  
A principal ideal $(\xi)$ for which $(\omega^{1/2})_{a^2}\subset (\xi)_{a^2}$ but for which $(\omega^{1/2}) _{a}\not\subset (\xi)_{a}$.\\
\end{example}

\begin{proof}[Construction]
We construct inductively an increasing sequence of positive integers $m_k$ starting with $m_1=0$ and define $\xi_1:=1$ and $\xi_j = \epsilon_k :=\frac{1}{m_k}$ for 
$m_k<j \le m_{k+1}$ for all $k > 1$.
Assume the construction up to $m_k$.  
Choose $m_{k+1} > e^{2m_k}$ and so that $(\xi_{a^2})_{m_{k+1}} \le \frac{1+ 1/k}{m_k}$.
This can be achieved using Equation (\ref{xiaa}) by choosing $m_{k+1}$ sufficiently large.
Then
\[
\frac{1}{m_k}= \xi_{m_{k+1}} \le (\xi_a)_{m_{k+1}}\le 
(\xi_{a^2})_{m_{k+1}} \le \frac{1+ 1/k}{m_k}.
\]
For $m_k < j \le m_{k+1}$ it follows from (\ref{xia}) that
\[
  \frac{m_k}{j}\big(\frac{1}{m_{k-1}} - \frac{1}{m_k}\big)+ \frac{1}{m_k}
 \le (\xi_a)_j
 \le \frac{m_k(1+\frac{1}{k-1})}{jm_{k-1}} + \frac{1}{m_k},
\]
from which follows, as $j$ increases, the asymptotic
\[
(\xi_a)_j \sim \frac{m_k}{jm_{k-1}} + \frac{1}{m_k}.
\]
In particular, if $j= [\frac{m_k^2}{m_{k-1}}]$, then $(\xi_a)_j \sim 
\frac{2}{m_k} $, while $(\omega^{1/2})_j \sim 
\frac{{m_{k-1}^{1/2}}}{m_k} $.
But then $\omega^{1/2} \ne O(\xi_a)$ and hence $(\omega^{1/2}) _{a}= 
(\omega^{1/2})\not\subset (\xi)_{a}$.
We now  show  that $\xi_{a^2} \ge \omega^{1/2} $ so $(\omega^{1/2}) 
_{a^2}=(\omega^{1/2})\subset (\xi)_{a^2}$.
When $m_k < j \le m_{k+1}$ and $k \ge 2$, by (\ref{xiaa}) and (\ref{e: Hn-Hm}),
\begin{align*}
j(\xi_{a^2})_j - j\omega^{1/2}_j &= {m_k}\Big((\xi_{a^2})_{m_k} - 
\frac{1}{m_k} +((\xi_a)_{m_k} - \frac{1}{m_k})(H_j-H_{m_k})\Big) 
+\frac{j}{m_k} -  j^{1/2}\\
&\ge  {m_k}(\frac{1}{m_{k-1}} - \frac{1}{m_k})(1 + 
\log\frac{j+1}{m_k+1}) + \frac{j}{m_k} -j^{1/2}\\
&\ge \frac{m_k}{2m_{k-1}}(1+ \frac{1}{2}\log\frac{j}{m_k}) + 
\frac{j}{m_k} - j^{1/2}.
\end{align*}
  Define the function $\phi(x) := \frac{m_k}{2m_{k-1}} 
\big(1+\frac{1}{2}\log\frac{x}{m_k}\big) + \frac{x}{m_k} -   x^{1/2}$ 
for $x\ge m_k$.
  Elementary calculus and the quadratic form in $\frac{1}{\sqrt x}$ of 
$\phi'$ with two real roots shows that the function $\phi$ has an absolute minimum on the interval 
$[m_k , \infty)$ at $x_{m_k}:= \big(\frac{m_k}{4}(1+\sqrt{1- \frac{4}{m_{k-1}}})\big)^2$ and a direct computation shows that $\phi(x_{m_k}) > 0$ for $k$ sufficiently large, because of the assumption that $m_k \ge e^{2m_{k-1}}$.
This proves that $\omega^{1/2}_j \le (\xi_{a^2})_j  $ for all $j$ and hence $(\omega^{1/2})\subset (\xi)_{a^2}$.
\ep

For equality cancellation we can slightly relax the sufficient conditions 
for the principal ideal case from those of the general case.

\bT {T: equality cancel}
\item[(i)]  Let $I = (\eta)$ be principal ideal, let $r(\eta_a)\asymp 
\log$  or $\eta \asymp \eta_a$,  and let $J$ be an arbitrary ideal. 
Then  $J_{a^2} = I_{a^2} \Rightarrow J_a = I_a$.
\item[(ii)] A sufficient condition on a general ideal $I$ for the second order am-equality cancellation implication in (i) to hold for arbitrary ideals $J$, is that every $\xi \in \Sigma(I)$ is dominated by some $\eta \in \Sigma(I)$ for which $r(\eta_a) \asymp \log$ or $\eta \asymp \widehat \eta$.
\eT

\bp
\item[(i)]   Whether  $r(\eta_a)\asymp \log$  or $\eta \asymp \eta_a$, $r(\eta_a)$ is equivalent to a monotone sequence, hence by 
Theorem \ref{T: monotone}(i), $J_a \subset I_a$. 
If $r(\eta_a)\asymp \log$, then by Theorem \ref{T: hat or log}(i) we can also conclude that $J_a \supset I_a$. 
If $\eta \asymp \eta_a$, i.e., $\eta$ is regular, then $\eta \le \xi_{a^2}$ for some $\xi\in\Sigma(J)$. 
But $\xi_{a^2}\in \Sigma ((\eta)_{a^2}) = \Sigma ((\eta))$, hence $\xi_{a^2}=O(\eta)$. 
Then $\xi_{a^2}\asymp \eta$ and thus $\xi_{a^2}$ is regular. 
It follows that $\xi$ is regular (cfr. Corollary \ref {C: g regular iff a(g)}) and hence $\eta = O(\xi)$. 
Thus $I\subset J$ and in particular, $I_a \subset J_a$.
\item[(ii)]  Follows combining  Theorem \ref{T: monotone}(i) and Theorem \ref{T: hat or log}(ii) and recalling that $\eta \asymp \widehat \eta$ implies that 
$\eta$ is regular and hence that $r(\eta_a)$ is equivalent to a monotone sequence.
\ep

\noindent We do not know if in (ii) the last condition $\eta \asymp \widehat \eta$ can be replaced by the more general condition of regularity, $\eta \asymp \eta_a$.\\

Because of Corollary \ref{C:higher order}, we can extend Theorems \ref{T: monotone},  \ref{T: hat or log}, and  \ref{T: equality cancel} to higher order arithmetic means.

\begin{theorem}\label{T: higher order cancellation}
Let $p \in \mathbb N$. If every $\co*$-sequence in the characteristic set $\Sigma (I)$ of an ideal $I$ is dominated by some $\eta$ in its characteristic set that satisfies the $p^{th}$ order exponential $\Delta_2$-condition $\underset{m} {\sup}\,\frac{m^2(\eta_{a^p})_{m^2}}{m(\eta_{a^p})_m} < \infty$, 
then $J_{a^{p+1}} \supset I_{a^{p+1}}$ implies $J_{a^p} \supset I_{a^p}$,
$J_{a^{p+1}} \subset I_{a^{p+1}}$ implies $J_{a^p} \subset I_{a^p}$, and $J_{a^{p+1}} = I_{a^{p+1}}$ implies $J_{a^p} = I_{a^p}$.
\end{theorem}


\begin{thebibliography}{99}
\bibitem{aC94}
Connes, A., \textit{Non Commutative Geometry,} San Diego Academic Press, 1994.

\bibitem{jC41}
Calkin, J. W., \textit{Two-sided ideals and congruences in the ring
of bounded operators in {H}ilbert space,}
Ann. of Math. (2) \textbf{42} (1941), pp.~839--873.

\bibitem{DFWW}
Dykema, K., Figiel, T., Weiss, .G and Wodzicki, M., \textit{The
commutator structure of operator ideals,}
Adv. Math., 185/1 pp. 1--79.

\bibitem{HLP52}
Hardy, G. H., Littlewood, J. E. and P\'olya, G., \textit{Inequalities.} 2d ed. Cambridge University Press, 1952.

\bibitem{GK69}
Gohberg, I. C. and Kre{\u\i}n, M. G., \textit{Introduction to the
Theory of Linear Nonselfadjoint Operators.} American Mathematical
Society (1969).

\bibitem{vKgW02}
Kaftal, V. and Weiss, G., \textit{Traces, ideals, and arithmetic
means,} Proc. Natl. Acad. Sci. USA \textbf{99} (11) (2002),
pp.~7356--7360.

\bibitem{vKgW04-Traces}
Kaftal, V. and Weiss, G., \textit{Traces on operator ideals and
arithmetic means,} preprint.

\bibitem{vKgW04-Soft}
Kaftal, V. and Weiss, G., \textit{Soft ideals and arithmetic mean
ideals,} IEOT, to appear.

\bibitem{vKgW04-Density}
Kaftal, V. and Weiss, G., \textit{$B(H)$ lattices, density, and 
arithmetic mean ideals,} preprint.

\bibitem{vKgW04-Majorization}
Kaftal, V. and Weiss, G., \textit{Majorization for infinite sequences, an extension of the Schur-Horn Theorem, and operator ideals,} in preparation.


\bibitem{nK89}
Kalton, N. J., \textit{Trace-class operators and commutators,} J.
Funct. Anal. \textbf{86} (1989), pp.~41--74.


\bibitem{MO79}
Marshall, A. W. and Olkin, I., \textit{Inequalities: Theory of 
Majorization and its Applications,} Academic Press Inc. [Harcourt 
Brace Jovanovich Publishers], Mathematics in Science and Engineering 
\textbf{143} (1979).


\bibitem{gW75}
Weiss, G., \textit{Commutators and Operators Ideals,} dissertation
(1975), University of Michigan Microfilm.


\bibitem{gW80}
Weiss, G., \textit{Commutators of {H}ilbert-{S}chmidt operators.
{I}{I},} IEOT \textbf{3} (4) (1980), pp.~574--600.

\bibitem{gW86}
Weiss, G., \textit{Commutators of {H}ilbert-{S}chmidt operators.
{I},} IEOT \textbf{9} (6) (1986), pp.~877--892.


\bibitem{jV89}
Varga, J., \textit{Traces on irregular ideals,} Proc. Amer. Math.
Soc. \textbf{107} 3 (1989), pp.~715--723.



\end{thebibliography}
\end{document}